\documentclass[12pt,a4paper]{amsart}

\usepackage[utf8]{inputenc}
\usepackage[margin=1in]{geometry}
\usepackage{hyperref}
\usepackage{amsmath}
\usepackage{amsfonts}
\usepackage{amsthm}
\usepackage{amssymb}
\usepackage{xcolor}
\usepackage{graphicx}
\usepackage{mathtools}

\hypersetup{
    colorlinks=true,
    linkcolor=blue,
    filecolor=magenta,      
    urlcolor=cyan,
}
 
\theoremstyle{plain}
\newtheorem{Th}{Theorem}[section]
\newtheorem{Cor}[Th]{Corollary}
\newtheorem{Prop}[Th]{Proposition}
\newtheorem{Lemma}[Th]{Lemma}

\theoremstyle{definition}
\newtheorem{Def}[Th]{Definition}
\newtheorem{Not}[Th]{Notation}
\newtheorem{Rmk}[Th]{Remark}

\newcommand{\IN}{{\normalfont \textrm{in}}}
\newcommand{\link}{{\normalfont \textrm{Link}}}
\newcommand{\del}{{\normalfont \textrm{Del}}}

\newcommand{\row}{{\normalfont \textrm{row}}}
\newcommand{\col}{{\normalfont \textrm{col}}}

\newcommand{\HT}{{\normalfont \textrm{ht}}}

\newcommand{\EN}{{\normalfont \textrm{EN}}}
\newcommand{\ENS}{{\normalfont \textrm{ENS}}}

\title{Generic Generalized Diagonal Matrices}

\begin{document}
\author{Vinh Nguyen and Hunter Simper}\thanks{Hunter Simper was partially supported by NSF grant DMS-2100288 and by Simons Foundation Collaboration Grant for Mathematicians \#580839.}
\address{Vinh Nguyen \\ Department of Mathematics \\ Purdue University \\
West Lafayette \\ IN 47907 \\ USA} 
\email{nguye229@purdue.edu}

\address{Hunter Simper \\ Department of Mathematics \\ Purdue University \\
West Lafayette \\ IN 47907 \\ USA} 
\email{hsimper@purdue.edu}
\maketitle
\begin{abstract}Generalized diagonal matrices are matrices that have two ladders of entries that are zero in the upper right and bottom left corners. The minors of generic generalized diagonal matrices have  square-free initial ideals. We give a description of the facets of their Stanley-Reisner complex. With this description, we characterize the configuration of ladders that yield Cohen-Macaulay ideals. In the special case where both ladders are triangles, we show that the corresponding complex is vertex decomposable. Also in this case, we compute the height and multiplicity of the ideals. \end{abstract}
\section{Introduction}

The initial ideals of determinantal ideals are square-free, hence a combinatorial approach, using Stanley-Reisner complexes, can be applied to study them. Herzog and Trung in \cite{Herzog1992} gave a description of the Stanley-Reisner complexes of generic determinantal ideals. They use this description to compute the Hilbert multiplicity of these ideals. We introduce a more general class of matrices, generalized diagonal matrices, for which the methods in \cite{Herzog1992} are applicable. Generalized diagonal matrices are matrices with two ladders of zeros in the bottom left and top right corner. We are able to extend the results of \cite{Herzog1992} to generalized diagonal matrices. Furthermore, as our main result, we classify the shapes of zeros which yield an ideal that is Cohen-Macaulay.
\begin{Def}
A \textit{generalized diagonal} (GD) matrix is a $n \times m$ matrix with two ladders of zeros, $L_1$ and $L_2$, in the bottom left and top right, respectively. The ladders are described as follows. Let $c_1\geq c_2 \geq ...\geq c_s$ be a non-increasing set of positive integers, with $s < m$ and $c_1 < n$, then $L_1$ consist of the last $c_i$ entries in column $i$. Similarly, let $d_t \leq d_{t+1} \leq ... \leq d_m$ be a non-decreasing set of positive integers, with $t >1 $ and $d_m < n$, then $L_2$ consist of the first $d_i$ entries in column $i$. We allow $L_1$ or $L_2$ to be empty.
\end{Def}

In section \ref{Sec-Groebner-Basis-GD} we will compute the initial ideals of the ideals of minors of generic GD matrices. In section \ref{Sec-Complex} we analyze their Stanley-Reisner complex. We give a description of the facets of these complexes, cf. \cite{Herzog1992},\cite{Miller2005}, and \cite{Conca2003}. Additionally, these complexes can also be realized as complexes associated to certain posets and have been studied in \cite[Section 7]{Bjorner1980}. \\

A special case of a generic GD matrix is one where there are triangles of zeros in the two corners. This is precisely when $c_1 = t_1, c_2 = t_1 - 1, ..., c_{t_1} = 1$ for some $t_1 \geq 0$ and $d_{m-t_2+1} = 1, d_{m-t_2+2} = 2,..., d_{m} = t_2$ for some $t_2 \geq 0$. These matrices are determined completely by their size and the parameters, $t_1$ and $t_2$, the sizes of the triangles of zeros. The Stanley-Reisner complex of the initial ideals of their minors behaves nicely. In section \ref{Sec-Dim-Purity}, using the description of the facets of these complexes, we are able to give an explicit formula for the height of the $r \times r$ minors in terms of $n,m,t$, and $r$. In section \ref{Sec-CM} we show that these ideals are Cohen-Macaulay. Finally in section \ref{Sec-Multiplicity} we give a formula for the multiplicity of these ideals. Two more special cases occur when $t_1 = t_2 = 0$ and $t_1 = n-1 = m-1, t_2 = 0$. These parameters describe generic matrices and generic upper triangular square matrices, respectively. Hence we recover the results in \cite{Herzog1992}.

\section{Generic GD Matrices}\label{Sec-Groebner-Basis-GD}

\begin{Not} \label{Notation-Matrix}
Throughout, $X = (x_{ij})$ denotes a generic $n \times m$ GD matrix with ladders of zeros $L_1$ and $L_2$. $I = I_r(X)$, the ideal of $r \times r$ minors of $X$, and $R = k[X]$. We order the non-zero entries of $X$ lexicographically, i.e., $x_{ij} > x_{kl}$ if $i < k$, or $i = k$ and $j < l$. We then extend this order to a lexicographic order on the monomials of $R$. \\

For convenience, after determining the shape of $L_1$ and $L_2$, we mean the non-zero entries of $X$ when we refer to entries of $X$.
\end{Not}

As is standard practice, to deduce properties about $I$, we first compute $\textrm{in}(I)$ and determine a Gröbner basis for it. After finding $\textrm{in}(I)$, we describe the facets of its Stanley-Reisner complex. Loosely speaking, the facets are union of stair shape paths starting next to $L_2$ and ending next to $L_1$.\\

We first need a lemma that describes how initial ideals in $R$ and in $S$ interact.

\begin{Lemma}\label{Lemma-GB-Lift} Let $S$ be a polynomial ring over a field and let $T$ be an ideal generated by indeterminates of $S$. Let $R = S/T$ and let ``$-$" denote images in $R$. Fix a monomial ordering on $S$ and consider the induced monomial ordering on $R$. Suppose $\{g_1,...,g_k\}$ is a Gröbner basis for a homogeneous ideal $I$ of $S$ such that $\overline{{\normalfont\IN}(g_i)} = {\normalfont\IN}(\overline{g_i})$ for all $i$ and ${\normalfont\IN}(I + T) = {\normalfont\IN}(I) + T$, then $\{\overline{g_1},...,\overline{g_k}\}$ is a Gröbner basis for $\overline{I}$.
\end{Lemma}

\begin{proof}
We only need to show that $ \big( \;\overline{\textrm{in}(g_i)}\; \big) = \textrm{in}(\overline{I})$, then we are done by the assumption $\big(\;\overline{\textrm{in}(g_i)}\;\big)  = (\textrm{in}(\overline{g_i}))$. Again, by that assumption, since $(\textrm{in}(\overline{g_i})) \subset \textrm{in}(\overline{I})$, we have $\big(\;\overline{\textrm{in}(g_i)}\;\big)\subset \textrm{in}(\overline{I})$. Equality follows if their Hilbert functions are the same. We compare their functions as follows, $$\textrm{HF}_{R/\textrm{in}(\overline{I})} = \textrm{HF}_{R/\overline{I}} = \textrm{HF}_{S/I+T} = \textrm{HF}_{S/\textrm{in}(I + T)} $$ $$= \textrm{HF}_{S/\textrm{in}(I) + T} = \textrm{HF}_{S/ ( \textrm{in}(g_i) ) + T} = \textrm{HF}_{R/ \left( \; \overline{\textrm{in}(g_i)} \; \right)}.$$ \end{proof}

\begin{Prop}\label{GB-I}
The minors of $X$ form a Gröbner basis for $I_r(X)$.
\end{Prop}

\begin{proof} Let $X'$ denote the fully generic $n \times m$ matrix, where all of the entries are non-zero. Let $S = k[X']$, $I' = I_r(X')$, and $T = (L_1 \cup L_2)$. Then we identify $R$ as $S/T$ and we have $I = I_r(X) = I' + T / T$. Further, we may choose an ordering on $S$ so that the ordering on $R$ is induced from that of $S$. Let $I = (\delta_1,...,\delta_s)$ where $\delta_i$ are the $r \times r$ minors of $X$. Let ${\delta_i}'$ be the lifts $\delta_i$ in $X'$, i.e., they are the determinants of the corresponding submatrix in $X$. By a well-known result, $\{{\delta_1}',...,{\delta_s}'\}$ forms a Gröbner basis for $I'$ \cite[5.3]{Conca2003}, \cite[2.4]{Herzog1992}. \\

We check that $\{{\delta_1}',...,{\delta_s}'\}$ satisfies the conditions of Lemma \ref{Lemma-GB-Lift}. By \cite[1.9]{Gorla2007}, $\textrm{in}(I' + T) = \textrm{in}(I') + T$ if for any minor ${\delta_i}' \in I'$ and any entry $x \in T$, there is an element $f \in I' \cap T$ such that $\textrm{in}(f) = \textrm{lcm}(x, \textrm{in}({\delta_i}'))$. Pick ${\delta_i}' \in I'$ and $x \in T$. Now either $\textrm{lcm}(x, \textrm{in}({\delta_i}')) = x \textrm{in}({\delta_i}')$ or $\textrm{lcm}(x, \textrm{in}({\delta_i}')) = \textrm{in}({\delta_i}')$, depending on whether or not $x$ is along the main diagonal of ${\delta_i}'$. In the former case, take $f = x{\delta_i}'$. In the latter case, ${\delta_i}'$ or its transpose is of the correct form to apply Lemma \ref{Lemma-Div-Det}, hence ${\delta_i}' \in (x) \cap T$ and we take $f = {\delta_i}'$. \\

We are done once we check that $ \textrm{in}(\delta_i) = \overline{\textrm{in}({\delta_i}')}$, where ``$-$" denotes image in $R$. Let $d_i,{d_i}'$ be the submatrices of $X,X'$ corresponding to $\delta_i,{\delta_i}'$, respectively. If $\delta_i \neq 0$, then by Lemma \ref{Lemma-Div-Det}, the main diagonal of ${d_i}'$ does not meet $T$. Hence the image of the main diagonal of $d_i$ is equal to the main diagonal of $\overline{{d_i}'}$ thus $ \textrm{in}(\delta_i) = \overline{\textrm{in}({\delta_i}')}$. If $\delta_i = 0$, then ${\delta_i}'\in T$, since $T$ is a monomial ideal, we have that each term of ${\delta_i}'$ is in $T$. In particular, $ \textrm{in}({\delta_i}')\in T$ so $\overline{\textrm{in}({\delta_i}')}=0$.
\end{proof}

\begin{Lemma}\label{Lemma-Div-Det} Let $Y$ be a $n\times n$ matrix of the form $\begin{bmatrix} * & * \\ A & * \end{bmatrix}$, where $A$ is a $(n-i+1)\times i$ submatrix. Then $\det Y\subset I_1(A)$. In particular, if $A$ is a submatrix of the form $\begin{bmatrix} 0 & x \\ 0 & 0 \end{bmatrix}$, where $x \neq 0$, then $x$ divides $\det Y$.
\end{Lemma}
\begin{proof}
Take a $n-i$ order expansion for the determinant of $Y$ along its last $n-i$ rows. Observe that in this expansion, all but one $n-i\times n-i$ minor is in $I_1(A)$, and that minor term corresponds to an $i \times i$ minor whose last row is the first row of $A$.
\end{proof}

\begin{Rmk}\label{Rmk-GB} The fact that the minors form a Gröbner basis is the cornerstone for the deductions in the following section. From the proof, we see that the placement of the zeros in a generic GD matrix ensures that diagonal terms of minors of $X'$ either survive in $X$ or correspond to a minor that is zero. If we attempt to generalize to sparse matrices, this approach breaks down as we cannot control the image in $X$ of the main diagonal terms of minors in $X'$. In \cite{Giusti1982}, the authors computed the height and classified the Cohen-Macaulayness of maximal minors of sparse matrices. In the case that a sparse matrix is, up to row and column operations, a GD matrix, we recover \cite[1.6.2]{Giusti1982} and extend their results to arbitrary size minors.

\end{Rmk}

\section{The Stanley-Reisner Complex of $I_r(X)$}\label{Sec-Complex}

We first establish some notation and definitions for some standard operations on simplicial complexes.

\begin{Def} Let $\Delta$ be a complex.
\begin{enumerate}
    \item Let $F$ be a face of $\Delta$, then the \textit{deletion} of $\Delta$ at $F$ is the subcomplex $$\del_F(\Delta) = \{ G \in \Delta : G \cap F = \emptyset\}.$$
    \item Let $F$ be a face of $\Delta$, then the \textit{link} of $\Delta$ at $F$ is the subcomplex $$\link_F(\Delta) = \{ G \in \Delta : G \cap F = \emptyset\ \text{ and } G \cup F \in \Delta\}.$$
    \item Let $\Sigma$ be a complex whose vertex set is disjoint from the vertex set of $\Delta$. Then the \textit{join} of $\Sigma$ with $\Delta$ is the following complex, $$\Delta * \Sigma = \{G \cup F \; : \; G \in \Delta, F \in \Sigma \}.$$
\end{enumerate}
\end{Def}

Next, we recall the Stanley-Reisner correspondence between square-free monomial ideals and simplicial complexes.

\begin{Def} $($\cite{Herzog2011}$)$
Given a square-free monomial ideal $\mathbf{a}\in k[x_1,...,x_n]$, let $\Delta_\mathbf{a}$ be the simplicial complex on the vertex set $x_1,...,x_n$ whose faces correspond to monomials which are not in $\mathbf{a}$. We call $\Delta_{\mathbf{a}}$ the \textit{Stanley-Reisner complex} of $\mathbf{a}$.
\end{Def}

By Proposition \ref{GB-I}, we have that the initial ideal of $I$ is a square-free monomial ideal, hence we introduce the following notation. 

\begin{Not}
For a generic GD matrix $X$, we set $\Delta_{I_r(X)}=\Delta_I:=\Delta_{\textrm{in}(I)}$, the Stanley-Reisner complex of $\IN(I)$. We will label the vertices of $\Delta_I$ using the entries of $X$. 
\end{Not}

We will abuse notation and will not distinguish between $x_{ij}$ as an entry of $X$ or as a vertex of $\Delta_I$. The meaning of $x_{ij}$ will be clear from context.

\begin{Def} A \textit{$k$-diagonal} is a set of entries of $X$ that form the main diagonal of some $k \times k$ submatrix of $X$. A \textit{$k$-chain} $x_1 < x_2 < ... < x_k$ is an ordered $k$-diagonal using the ordering on $R$.

\end{Def}

Note that a face $C$ on the vertex set $\{x_{ij}\}$ is a face of $\Delta_I$ if and only if the monomial support of $C$ is not in $\textrm{in}(I)$. That is to say, a face $C$ of $\Delta_I$ corresponds to a set of entries of $X$ whose product is not divisible by the leading term of a $r\times r$ minor of $X$. This is precisely the condition that $C$ does not contain any $r$-diagonals. This leads to the following definition.

\begin{Def}[F-Condition] \label{Def-Fk}
Let $k$ be a positive integer, then we say that a set of entries $C$, satisfies condition $F_k$ if no subset of $C$ is a $k$-diagonal. Note that if $C$ satisfies condition $F_k$, then $C$ also satisfies condition $F_{k+1}$. The faces of $\Delta_I$ correspond to exactly those sets that satisfy condition $F_k$ for some $k \leq r$.
\end{Def}

Next, we describe a technical process, which yields a particular subset of $C$, that allows for induction on condition $F_k$. Loosely speaking, this process scrapes off the top perimeter of $C$. If $C$ satisfies condition $F_k$, then after scraping, what remains will satisfy condition $F_{k-1}$.

\begin{Def} \label{Def-Order} We define a new ordering on the entries of $X$ in the following way, $$x \succ y \textrm{ iff }  \textrm{col}(x) < \textrm{col}(y) \textrm{ or  col}(x) = \textrm{col}(y) \textrm{ and row}(x) > \textrm{row}(y).$$ If $C$ is empty, then set $S(C) = C$. Otherwise, suppose $C$ is not empty, we inductively construct a subset $S(C) = \{y_1,...,y_s\} \subset C$ as follows. Choose $y_1 = \textrm{max}_{\succ}\{y \in C \}$. After choosing $y_i$, let $$B(y_i) = \{ y\neq 0 : \textrm{row}(y) \leq \textrm{row}(y_i) \textrm{ and col}(y) \geq \textrm{col}(y_i) \} \setminus \{y_i\}. $$ If $B(y_i)\cap C = \emptyset$ stop, otherwise choose $y_{i+1} = \textrm{max}_{\succ} \{ y \in B(y_i)\cap C\}$. Notice by construction, the elements of $S(C)$ are ordered in the following way, $y_1 \succ y_2 \succ ... \succ y_s$. We also have $B(y_1) \supset B(y_2) \supset ... \supset B(y_s)$, and each element in $B(y_i)$ is smaller than $y_i$.
\end{Def}
For example if $C=\{a,\cdots, m\}$ are in the positions indicated below, then $S(C)=\{j,k,l,h,c,d,a,b\}$ (bold entries). 
\[\begin{bmatrix}
 & & & & \mathbf{a}& \mathbf{b}\\
 & & & \mathbf{c}& \mathbf{d}& e\\
 & & & & f&g \\
 & & &\mathbf{h} &i & \\
\phantom{x_5} &\mathbf{j} & \mathbf{k} & \mathbf{l} &m & 
\end{bmatrix}.\]

We will refer to the above process as \textit{scraping}. Next, we prove some statements about $S(C)$.

\begin{Prop}\label{Prop-Scrape} Let $C$ be a set of entries of $X$.
\begin{enumerate}
    \item $S(C)$ satisfies condition $F_2$ and is the top left perimeter of $C$ in the following sense. For a non-zero entry $x$, if for some $y \in S(C)$, we have ${\normalfont \textrm{row}}(x) \leq {\normalfont \textrm{row}}(y)$ and ${\normalfont \textrm{col}}(x) \leq {\normalfont \textrm{col}}(y)$, then $x \notin C \setminus S(C)$. In particular, this property says that $S(C)$ is closed under ``going up in a column" or ``going left in a row". Meaning for any $y \in S(C)$ and any $x \in C$, if ${\normalfont \textrm{row}}(x) \leq {\normalfont \textrm{row}}(y)$ and ${\normalfont \textrm{col}}(x) = {\normalfont \textrm{col}}(y)$, then $x \in S(C)$, or if ${\normalfont \textrm{row}}(x) = {\normalfont \textrm{row}}(y)$ and ${\normalfont \textrm{col}}(x) \leq {\normalfont \textrm{col}}(y)$, then $x \in S(C)$. 
    \item For $k \geq 2$, if $C$ satisfies condition $F_k$, then $C \setminus S(C)$ satisfies condition $F_{k-1}$.  Conversely, if $C$ does not satisfy condition $F_k$, then $C \setminus S(C)$ also does not satisfy condition $F_{k-1}$.
\end{enumerate}
\end{Prop}
\begin{proof}
(1) It is evident from the construction of $S(C)$ that it satisfies condition $F_2$, as each $B(y_i)$ consist entirely of entries to the right and up from $y_i$. \\

Now we show the statement about $S(C)$ being the top left perimeter of $C$. Let $x$ be an entry satisfying the property in the proposition. We may assume $x \notin S(C)$, and we will show then that $x \notin C$. Consider the set $E = \{ y \in C : y \succ x \}$. We may assume $E \cap S(C)$ is not empty. For if it is empty, then $y_1$, the maximal element of $C$, is not in $E$, which means $x \succ y_1$, hence $x \notin C$. Next, we will show that for any $y \in E \cap S(C)$ that $x \in B(y)$. Let $z \in S(C)$ be minimal with respect to $\succ$ such that $\row(x) \leq \row(z)$ and $\col(x) \leq \col(z)$. By minimality of $z$, we have $y \succeq z$, and since $z \in S(C)$, we have $\row(z) \leq \row(y)$. Hence $\row(x) \leq \row(y)$. As $y \in E$, we have that $x \in B(y)$. \\

Next, let $y_t$ be the minimal element of $E \cap S(C)$. From above, we know $x \in B(y_t)$. If $y_t$ is the last element of $S(C)$, then $x \notin C$ as $B(y_t) \cap C = \emptyset$. Hence we may assume there is a $y_{t+1} \in S(C)$. By minimality of $y_t$, we have $x \succ y_{t+1}$. Finally, assume that $x \in C$, then $x \in B(y_t)\cap C$, but then $ y_{t+1} \succeq x$ so that $x = y_{t+1}$ which is a contradiction. The in particular follows immediately. \\ \newline
(2) Suppose $C\setminus S(C)$ does not satisfy condition $F_{k-1}$, then there is a $(k-1)$-chain $x_1 < ... < x_{k-1} := x$ in $C\setminus S(C)$. We want to find an element, $w \in S(C)$, such that $\col(w) < \col(x)$ and $\row(w) < \row(x)$. This would yield a $k$-chain $x_1 <... < x_{k-1} < w$ in $C$ which would be a contradiction. Consider $E = \{ y \in S(C) : y \succ x\}$, and let $y_t$ be the minimal element of $E$. By $(1)$, it must be that $\col(y_t) < \col(x)$, otherwise $x \in S(C)$. If $\row(y_t) < \row(x)$, we are done. Hence assume $\row(y_t) \geq \row(x)$, but we will see that this leads to a contradiction. Notice that $x \in B(y_t)\cap C$, as $x \notin S(C)$, hence we must have that $y_{t+1} \succ x$, but since $y_{t} \succ y_{t+1}$, this contradicts the minimality of $y_t$.\\

For the converse, let $x_1< ... <x_k$ be a $k$-chain in $C$. Then since $S(C)$ satisfies condition $F_2$, there can be at most one $x_i \in S(C)$, hence $C \setminus S(C)$ contains at least a $(k-1)$-chain.
\end{proof}
\begin{Def} Let $C$ be a set of non-zero entries of $X$, we say that $C$ is a \textit{stair} if it satisfies condition $F_2$. We now use stairs to define a covering condition on sets of non-zero entries. For $k\geq 1$, we say that $C$ is a \textit{$k$-stair}, if we can write $C = \bigcup_{j=1}^k S_j$, where each $S_j$ is a stair, and $k$ is as small as possible. For convenience, for $k \leq 0$, a $k$-stair is the empty set. Note that a non-empty set of non-zero entries of $X$ is a $k$-stair for exactly one $k$. We say that $C$ is a maximal $k$-stair if it's maximal with respect to inclusion and being a $k$-stair.\end{Def}
Returning to the example before \ref{Prop-Scrape}, $C=\{a,\cdots, m\}$ is a non-maximal $2$-stair. Clearly $C$ does not satisfy $F_2$, consider for instance $\{a,e\}\subset C$, so we need to demonstrate that $2$ stairs are sufficient to cover $C$. Take $S_1=S(C)=\{j,k,l,h,c,d,a,b\}$  and $S_2=\{m,i,f,g,e\}$. Note that $S_1$ and $S_2$ are not unique choices. Moreover $C\cup \{x_1,\cdots, x_5\}$ is a maximal $2$-stair containing $C$.
\[\begin{bmatrix}
 & & & & a& b\\
 & & & c& d& e\\
 & & & x_1 & f&g \\
 x_2 & x_3 & x_4 &h &i & \\
x_5 &j & k & l &m & 
\end{bmatrix}.\]

\begin{Prop} \label{Prop-Stairs}
Let $C$ be a set of entries of $X$.
\begin{enumerate}
    \item If $C$ is a $k$-stair, then $C$ satisfies condition $F_{k+1}$.
    \item If $C$ satisfies condition $F_k$, then $C$ is contained in a $j$-stair for some $j < k$. In particular, if $C$ maximally satisfies condition $F_{k+1}$ but does not satisfy condition $F_{k}$, then $C$ is a maximal $k$-stair.
    \item If $C$ is a maximal $k$-stair, then $C \setminus S(C)$ is a maximal $(k-1)$-stair in the submatrix obtained by removing the first row and column. In particular, if $C$ is a non-empty maximal $k$-stair, then $C$ maximally satisfies condition $F_{k+1}$ but does not satisfy condition $F_k$.
    \item All faces of $\Delta_I$ are contained in a $j$-stair for some $j < r$. In particular, $C$ is a facet of $\Delta_I$ if and only if $C$ is a maximal $(r-1)$-stair.
    \item If $C$ is a stair, then $S(C) = C$.
    \item If $C$ is a non-empty maximal $k$-stair, then $S(C)$ is a maximal stair.
\end{enumerate}
\end{Prop}

\begin{proof}
(1) Choose $k+1$ entries of $C$. By the pigeonhole principle, at least two of the chosen entries are contained in a stair $S$. Since $S$ satisfies condition $F_2$, they cannot be part of a $2$-diagonal, let alone a $(k+1)$-diagonal. \\

(2) We proceed by induction on $k$. The base case, when $k=2$, follows from the definition of a stair.\\

For the inductive step, suppose $C$ satisfies condition $F_k$ for some $k > 2$. We apply Proposition \ref{Prop-Scrape} to obtain subsets $S(C)\subset C$ and $C'= C\setminus S(C)$, where $S(C)$ satisfies condition $F_2$ and $C'$ satisfies condition $F_{k-1}$. By induction, $C'$ is contained in $S$, a $j$-stair for some $j < k-1$, and $S(C)$ is a stair. Hence $C = C' \cup S(C)$ is contained in $S \cup S(C) $, a $l$-stair for some $l < k$. \\ 

The in particular follows from the statement we just proved and $(1)$. \\

(3) We apply Proposition \ref{Prop-Scrape} to write $C = S(C) \sqcup C'$, where $S(C)$ satisfies condition $F_2$ and $C' = C \setminus S(C)$ satisfies condition $F_{k}$. Also by Proposition \ref{Prop-Scrape} (1), $C'$ does not meet the first row or column. Thus, we can consider $C'$ in $Y$, the submatrix obtained by removing the first row and column. Since $C'$ satisfies condition $F_{k}$, by (2), it is contained in a maximal $j$-stair, $E$, of $Y$ for some $j < k$. Back in $X$, we have $C \subset S(C) \cup E$. Now $S(C) \cup E$ satisfies condition $F_{k+1}$, otherwise, if it contains a $(k+1)$-diagonal, then at most one of the entries in the diagonal can be in $S(C)$. This implies that $E$ contains at least a $k$-diagonal, a contradiction. As $C$ is maximal, we have $C = S(C) \cup E$. Since $S(C)$ is a stair, $E$ must be a $(k-1)$-stair, otherwise $C$ could be written as a union of less than $k$ stairs.\\

It remains to show that back in $X$, $E \cap S(C) = \emptyset$ as then we would have $E = C \setminus S(C)$. Suppose not, then let $z = x_{ab}$ be the minimal element of $E \cap S(C)$ with respect to $\succ$. Since $E$ is contained in $Y$, there exists an entry $w = x_{(a-1)(b-1)}$ directly to the right and up from $z$. By Proposition \ref{Prop-Scrape}, $S(C)$ is the top perimeter of $C$, hence $w \notin C$. We will show that $A = C \cup \{w\}$ satisfies condition $F_{k+1}$, then using $(2)$, this will contradict the assumption that $C$ is a maximal $k$-stair. Suppose not, then $A$ contains a $(k+1)$-chain, $x_1 < ... < x_{k+1}$. Now this $(k+1)$-chain must contain $w$, otherwise it would be a $(k+1)$-chain completely contained in $C$, which contradicts $C$ satisfying condition $F_{k+1}$. As $w$ is above $S(C)$, we must have $x_{k+1} = w$. Now, we can replace $x_{k}$ with any element in $C$, whose row and column index is greater than that of $w$ and less than or equal to that of $x_{k}$, and preserve the chain. But $z$ is such an element, hence we may assume $x_{k} = z$. Now as $S(C)$ satisfies condition $F_2$, $x_1,...,x_{k-1}$ must be in $E$. But then $x_1 < ... <x_{k-1}$ is a $(k-1)$-chain in $E$, which is a contradiction as $E$ satisfies condition $F_{k-1}$. \\

For the in particular, from $(1)$ we only need to show that $C$ does not satisfy condition $F_k$. But this follows from induction and Proposition \ref{Prop-Scrape} (2).  \\

(4) Follows immediately from (2) and (3). \\

(5) Suppose $C \setminus S(C)$ is not empty; then by Proposition \ref{Prop-Scrape}, $C \setminus S(C)$ satisfies condition $F_{1}$, a contradiction.   \\

(6) Let $E$ be a maximal stair containing $S(C)$. We first show that $E \subset C$. By (3), $C \setminus S(C)$ is a $(k-1)$-stair considered in $X$. Let $C' = E \cup (C \setminus S(C))$, clearly $C \subset C'$. Now $C'$ is a union of $k$ stairs, for contradiction, suppose it is a $j$-stair for some $j < k$. Then $C'$ would satisfy condition $F_k$, but this would imply that $C$ satisfies condition $F_k$ which contradicts (3). Hence $C'$ is a $k$-stair. As $C$ is a maximal $k$-stair, $C = C'$ and hence $E \subset C$.\\

Next, we show that $E = S(C)$. For contradiction, suppose $E \setminus S(C)$ is not empty. Using the notation in the scraping process, let $S(C) = \{y_1,...,y_u\}$, and since $S(E) = E$, we may also write $E = \{z_1,...,z_s\}$. Let $z_l$ be the maximal element in $E \setminus S(C)$, we claim that $z_l \notin C$ which leads to a contradiction. Now after ordering $S(C) \cup \{z_l\}$ using $\succ$, there are two edge cases depending on if $z_l$ is the maximal or minimal element. If $z_l$ is the maximal element, then $z_l \succ y_1$, and since $y_1$ is the maximal element of $C$, the claim follows. If $z_l$ is the minimal element, then by construction of $S(C)$, we have $B(y_u)\cap C = \emptyset$. Since $z_l \in B(y_u)$, the claim follows.\\

Next, assume we are outside the two edge cases, in particular, $l \neq 1$ and $l \neq s$. By maximality of $z_l$, we have that $z_{l-1}$ is in $S(C)$. Write $z_{l-1} = y_t$ for some $y_t \in S(C)$. Then $z_l \in B(z_{l-1}) = B(y_t)$. Assume for contradiction that $z_l \in C$. Then $B(y_t) \cap C \neq \emptyset$. By construction of $S(C)$, there is a $y_{t+1} \in S(C)$ with $y_{t+1} \succeq z_l$. Now $y_{t+1} \in B(z_{l-1}) \cap E$, but $z_l$ is the maximal element in $B(z_{l-1}) \cap E$, hence $z_l = y_{t+1}$. This contradicts $z_l \notin S(C)$.\end{proof}

The scraping process and Proposition \ref{Prop-Stairs} (5) justifies the following description of a maximal stair. Start with an entry $x$, where $B(x)$ is empty. Then at each step pick an entry either down or to the left from the previous entry, proceed until there are no more entries down or to the left. In other words, maximal stairs are formed by maximal paths winding downwards and leftwards. This description is inverse to the scraping process, hence the proposition justifies this description of a stair, cf. \cite[Section 3]{Herzog1992}. \\

We now give a more precise description of what the facets look like.
\begin{Not}
Let $X$ be a matrix. Set $U_k(X)$ to be the set of non-zero entries in the $k \times k$ triangle in the upper right corner of $X$, i.e., $U_k(X) = \{x_{ij} \neq 0 : j - i \geq m-k \}$. 
Set $D_k(X)$ be the set of non-zero entries in the $k \times k$ triangle in the lower left corner of $X$, i.e., $D_k(X) = \{x_{ij} \neq 0: i - j \geq n-k \}$. Note that $U_k(X)$ or $D_k(X)$ may be empty.
If the matrix $X$ is clear from context we omit $X$ and just write $U_k$ and $D_k$ \end{Not}

\begin{Prop} \label{Prop-Contain-Triangle}
Let $C$ be a set of entries that satisfies condition $F_k$, then $C \cup U_{k-1} \cup D_{k-1}$ satisfies condition $F_k$. In particular, if $C$ is a facet, then $C$ contains $U_{r-1}$ and $D_{r-1}$.
\end{Prop}
\begin{proof}
This is clear as any entry in $U_{k-1} \cup D_{k-1}$ cannot be part of any $k$-diagonal. The in particular is also clear as $C$ is maximal with respect to satisfying condition $F_r$.
\end{proof} 

This proposition shows that if we let $X'$ be the matrix obtained from $X$ by zeroing the entries in $U_{r-1}$ and $D_{r-1}$, then $\Delta_I=\Delta_{I_r(X')} * (U_{r-1}\cup D_{r-1})$. Since $U_{r-1}\cup D_{r-1}$ is a simplex, the join is the same as an iterated cone over the vertices of $U_{r-1}\cup D_{r-1}$. Hence if we are concerned with any property that is preserved by taking cones, such as purity, we may replace $X$ with $X'$.\\

A $k$-stair $C$ can be described as the triangles $U_{k-1}$ and $D_{k-1}$ along with $k$ disjoint tendrils going from $U_{k-1}$ to $D_{k-1}$. These tendrils come from a particular stair decomposition of $C$. This is the content of the next proposition.

\begin{Prop} \label{Prop-Tendrils}
Let $C$ be a maximal $k$-stair. Then there exist a stair decomposition $C = \bigcup_{j=1}^{k} S_j$ such that $T_j = S_j \setminus (U_{k-1} \cup D_{k-1})$ are pairwise disjoint.
\end{Prop}
\begin{proof}
We proceed by induction on $k$. The base case is $k=1$ and is immediate. For the induction step, set $S_1 = S(C)$. Then by Proposition \ref{Prop-Stairs} (3), $S\setminus S(C)$ is a maximal $(k-1)$-stair in the matrix $Y$, obtained by deleting the first row and column. Hence by induction, there is a stair decomposition $S \setminus S(C) = \bigcup_{j=2}^{k} S_j$ where the $T_j$'s are pairwise disjoint. Hence $S =\bigcup_{j=1}^{k} S_j$ gives the desired stair decomposition.
\end{proof}

We will call the sets $T_j$ the \textit{tendrils} of a $k$-stair. Notice that the proof shows that a stair decomposition of a $k$-stair can be obtained by repeated scraping. The facets of $\Delta_I$ can be described as a union of $r-1$ disjoint tendrils along with $U_{r-1}$ and $D_{r-1}$. \\

\section{Dimension and Purity of $\Delta_I$}\label{Sec-Dim-Purity}

Now that we have description of the facets of $\Delta_I$, we can characterize when $\Delta_I$ is pure based on the shape of the ladders of zeroes, $L_1$ and $L_2$, in $X$. In the case that $L_1$ and $L_2$ are triangles, there is a simple formula for the dimension of $\Delta_I$. The formula is in terms of the height of the ideal of minors of a completely generic matrix, and a correction term corresponding to heights of ideals of minors of generic symmetric matrices. \\

Before proceeding, we setup some definitions. 

\begin{Def} A GD matrix is \textit{unpinched} if there is no contiguous $2 \times 2$ submatrix of the form $$\begin{bmatrix} * & 0 \\ 0 & * \end{bmatrix}.$$ Any GD matrix can be decomposed as a block diagonal matrix where the diagonal blocks are unpinched GD matrices. \\

We call a non-zero entry $x$ a \textit{corner} of $L_2$, if $B(x)$ (see \ref{Def-Order}) is empty. We call a non-zero entry $x$ a corner of $L_1$, if after transposing $X$, $B(x)$ is empty.
\end{Def}

\begin{Prop}\label{Prop-Unpinched-Purity} Let $X$ be an unpinched generic GD matrix. If $\Delta_I$ is pure, then the corners of $L_1$ are on the same diagonal or $L_1$ is contained in a $r-1 \times r-1$ triangle, and the corners of $L_2$ are on the same diagonal or $L_2$ is contained in a $r-1 \times r-1$ triangle. The converse holds when $r=2$, even if $X$ is pinched.
\end{Prop}

\begin{proof}
We may assume $L_1$ and $L_2$ contain triangles of size $r-1\times r-1$. Since doing so does not change any assumptions by Proposition \ref{Prop-Contain-Triangle}. \\

%%Small note: It may change the assumption that $X$ is unpinched. But if $X$ is pinched after zeroing out $D_{r-1}$ and $U_{r-1}$, then it must be a diagonal matrix, in which case the conclusion is still true.

We first deal with the $r=2$ case. In this case, we will also show that the converse holds. When $r=2$, the facets of $\Delta_I$ are maximal stairs. The dimension of any maximal stair starting at $(a,b)$ and ending at $(c,d)$, where $(a,b) \succ (c,d)$, only depends on its end points and is equal to $$c - a + b - d + 1.$$ 

Assume $\Delta_I$ is pure. Let $x = (x_1,x_2)$ and $y = (y_1,y_2)$ be adjacent corners of $L_1$. Since $X$ is unpinched, there exist a corner $z = (z_1,z_2)$ of $L_2$ such that there exist two maximal stairs $S$ and $T$ that both start at $z$ and end at $x$ and $y$, respectively. As $\Delta_I$ is pure, the dimensions of $S$ and $T$ are equal, hence $$x_1 - z_1 + z_2 - x_2 + 1 = y_1 - z_1 + z_2 - y_2 + 1.$$
This shows $x_1 - x_2 = y_1 - y_2$ and hence they are along the same diagonal. For $L_2$, we can repeat the same calculation where we pick two maximal stairs $S$ and $T$ that end at the same point but start at two adjacent corners of $L_2$. \\

Conversely, assume that the corners of $L_1$ and $L_2$ are on the same diagonal. Any maximal stair must start at a corner of $L_2$ and end at a corner of $L_1$. But for any corner $(x_1,x_2)$ of $L_1$, we have, by assumption, that $x_1 - x_2$ is constant. Similarly, $y_1 - y_2$ is constant for $(y_1,y_2)$, a corner of $L_2$. Finally, the dimension of any maximal stair is a function of $x_1 - x_2$ and $y_1 - y_2$, hence their dimensions are equal. Notice that this argument does not require $X$ to be unpinched. \\

Now assume $r > 2$. We first show that $L_1$ and $L_2$ are not rectangles, i.e., they have more than 2 corners each. We induct on $r$. For $r=3$, suppose $\Delta_I$ is pure and for contradiction, assume that $L_1$ is a rectangle. Let $E$ be the set of all non-zero entries in the first row and column of $X$. Notice $E$ is a face of $\Delta_I$. We have that $\link_{E}(\Delta_I)=\Delta_{J}$, where $J$ is the ideal of $r-1$ minors of $Y$, the submatrix obtained by deleting the first row and column. Since $\Delta_I$ is pure, we have that $\Delta_J$ is pure. Hence from the $r=2$ case, we have that $L_1$ is a square after removing its first column. Hence $L_1$ is a $k\times k+1$ rectangle for some $k$. Similarly, let $E'$ be the set of all non-zero entries in the last row and column of $X$, then $E'$ is also a face of $\Delta_I$. We get that $\link_{E'}(\Delta_I)=\Delta_{J'}$, where $J'$ is the ideal of $r-1$ minors of the submatrix $Y'$, obtained from deleting the last row and column. By the same reasoning as before, we get that $L_1$ is a $l+1\times l$ rectangle for some $l$, contradicting the previously shown form of $L_1$.\\

For the inductive step, note that if $L_1$ is a rectangle, then either $L_1$ minus the first column or $L_1$ minus the last row is a rectangle, hence $\Delta_J$ or $\Delta_J'$ is not pure by the inductive hypothesis.\\

Repeat this argument with $L_2$ to get that $L_2$ cannot be a rectangle.\\

Now we compare $L_1$ in $Y$ and in $Y'$. Since $L_1$ has more than two corners, $L_1$ considered in $Y$ shares at least one corner with $L_1$ considered in $Y'$. Hence by the inductive hypothesis, all of the corners of $L_1$ are along the same diagonal when considered in $X$. The same argument holds for $L_2$.
\end{proof}
\begin{Prop}\label{Prop-Pure-DisjointTendril}
If all tendrils of all facets of $\Delta_I$ are of the same dimension, then $\Delta_I$ is pure. Conversely, suppose $I \neq 0$, $X$ is unpinched, and $\Delta_I$ is pure, then all tendrils of all facets of $\Delta_I$ are of the same dimension.

\end{Prop}

\begin{proof}
The first statement is immediate from Proposition \ref{Prop-Tendrils}. \\

For the converse, we proceed by induction on $r$. We may assume $L_1$ and $L_2$ contain triangles of size $r-1\times r-1$. The base case when $r=2$ follows from Proposition \ref{Prop-Unpinched-Purity}. For the inductive step, let $C$ be a facet of $X$. Then as usual, we can write $C = S(C) \cup (C \setminus S(C))$, where $C \setminus S(C)$ is a facet of $\Delta_{I_{r-1}(Y)}$, where $Y$ is the submatrix obtained by removing the first row and column. By induction, the tendrils of $C \setminus S(C)$ all have the same dimension. By the proof of Proposition \ref{Prop-Tendrils}, the tendrils of $C$ are $T = S(C)$ along with the tendrils of $C \setminus S(C)$. \\

Now let $D$ be the last row and column and let $E$ be the union of the first $r-1$ rows and columns along with $D$. It is clear that $E$ is a facet of $\Delta_I$ and that $D$ is a tendril of $E$. By the same argument above, $D$ has the same dimension as any tendril of $E \setminus S(E)$ and hence it also has the same dimension as any tendril of $C \setminus S(C)$. By construction, both $D$ and $S(C)$ start and end at corner positions of $L_1$ and $L_2$, respectively. Hence by Proposition \ref{Prop-Unpinched-Purity}, since the corners of $L_1$ and $L_2$ are on the same diagonal, we have that $D$ and $S(C)$ have the same dimension.
\end{proof}

The converse is false if $X$ is pinched as shown by considering the $3 \times 3$ minors of the following matrix $$\begin{bmatrix} x_{11} & x_{12} & 0  \\ 
                  0 & x_{22} & 0  \\ 
                  0 & 0 &  x_{34}  
\end{bmatrix}.$$ The maximal $2$-stairs are $$\{x_{11},x_{12},x_{22}\},\{x_{11},x_{12},x_{34}\}, \textrm{ and } \{x_{12},x_{22},x_{34}\}.$$ Now consider $\{x_{11},x_{12},x_{22}\}$. Its two tendrils, obtained by repeatedly scraping, are $\{x_{11},x_{12}\}$ and $\{x_{22}\}$. The two tendrils clearly do not have the same dimension.

\begin{Prop} \label{Prop-Purity}
Set $X'$ to be the matrix obtained from $X$ by zeroing $U_{r-1}$ and $D_{r-1}$ and let $\Delta'$ be the complex associated to $I_r(X')$. Let $A_i$ be the diagonal blocks of an unpinched decomposition of $X'$. Assume $I \neq 0$.

\begin{enumerate}
    \item If $\Delta'$ is pure, then $\Delta_{I_2(X')}$ is pure.
    \item $\Delta_{I_2(X')}$ is pure if and only if $\Delta_{I_2(A_i)}$ are pure of the same dimension.
\end{enumerate}

\end{Prop}

\begin{proof}
(1) Suppose $X'$ is unpinched, then the statement follows by Proposition \ref{Prop-Unpinched-Purity}. Hence we may assume $X'$ is pinched. We induct on $r$, the base case is $r=2$ and is trivial. Now suppose $r > 2$. As before, let $E$ and $F$ by the first row and column and last row and column, respectively. Set $Y$ and $Z$ to be the submatrix obtained by deleting $E$ and $F$, respectively. Then we can consider $\link_{E}(\Delta_I)$ and $\link_{F}(\Delta_I)$ which correspond to the complex associated to the $r-1$ minors of $Y$ and $Z$, respectively. By induction, $\Delta_{I_2(Y)}$ and $\Delta_{I_2(Z)}$ are pure. Finally, notice that any stair in $X'$ is completely contained in $Y$ or $Z$, as otherwise $X$ would be unpinched. Hence every stair in $X'$ are the same dimension so that $\Delta_{I_2(X')}$ is pure. \\

(2) Any stair in $\Delta_{I_2(X')}$ corresponds to a stair in some $\Delta_{I_2(A_i)}$. In fact, $\Delta_{I_2(X')} = \bigsqcup \Delta_{I_2(A_i)}$. The statement follows immediately.
\end{proof}

We finish this section by working out a formula for the dimension of $\Delta_I$, and thus the height of $I$, in the special case where $L_1$ and $L_2$ are triangles. In the next section, we will show that $L_1$ and $L_2$ being triangles is equivalent to $\Delta_I$ being Cohen-Macaulay. \\

Recall that if $Y$ is a $n \times m$ generic matrix, then $$\textrm{ht} \; I_r(Y) = (n - r + 1)(m - r + 1) =: \textrm{EN}(n,m,r),$$ as shown in \cite{EagonNorthcott}. Also, we define $$\textrm{ENS}(n,r) := \binom{n-r+2}{2},$$ which is the height of $I_r(Y)$ when $Y$ is a generic symmetric matrix \cite{Jozefiak}.

\begin{Th} \label{Th-Triangle-Dim}
Let $L_1$ be a triangle of size $t_1$ and $L_2$ be a triangle of size $t_2$. Then $\Delta_I$ is pure with $$\dim \Delta_I = \dim R - \EN(n,m,r) + \ENS(t_1,r) + \ENS(t_2,r) - 1,$$ and thus $$\HT \; I = \EN(n,m,r) - \ENS(t_1,r) - \ENS(t_2,r).$$
\end{Th}

\begin{proof}
It suffices to show that the dimension of each facet is $$\dim R - \EN(n,m,r) + \ENS(t_1,r) + \ENS(t_2,r) - 1.$$ 
A facet of $\Delta_I$ is a $(r-1)$-stair, so we count the size of these stairs. \\

We induct on $r$. The base case is $r=2$. By Proposition \ref{Prop-Unpinched-Purity}, $\Delta_I$ is pure. Hence we need only count the size of a specific maximal stair, namely the first row and column. One computes that the size of the first row and column is $$n + m - 1 - t_1 - t_2 = \dim R - \EN(n,m,1) + \ENS(t_1,1) + \ENS(t_2,1).$$

For the inductive step, we apply Proposition \ref{Prop-Stairs} (3) to write an $(r-1)$-stair $C$ as $C = S(C) \sqcup C'$, where $C'$ is a maximal $(r-2)$-stair in the submatrix obtained by removing the first row and column. $S(C)$ is a maximal stair by Proposition \ref{Prop-Stairs} (6), hence it is the same size as the first row and column. Let $R'$ be the ambient ring of the submatrix, then notice that $\dim R = |S(C)| + \dim R'$. Now we apply the inductive hypothesis to get $$|C| = |S(C)| + |C'| = $$
$$|S(C)| + \; \dim R' -  \EN(n-1,m-1,r-1) + \ENS(t_1-1,r-1) + \ENS(t_2-1,r-1) $$
$$= \dim R - \EN(n,m,r) + \ENS(t_1,r) + \ENS(t_2,r).$$ \end{proof} 

\section{Cohen-Macaulayness of $I_r(X)$}\label{Sec-CM}

In this section, we will characterize the configurations of zeros of a generic GD matrix $X$ so that $I = I_r(X)$ is Cohen-Macaulay. We say that a complex is Cohen-Macaulay if its corresponding ideal is Cohen-Macaulay. As we have shown in section $2$, under the lexicographical ordering, $\textrm{in}(I)$ is square-free, hence we may apply the following. 

\begin{Th}\label{Th-ICM-iff-InICM}$($\cite[1.3]{Conca2020}$)$ Let $J$ be a homogeneous ideal such that $\normalfont{\textrm{in}}(J)$ is square-free. Then $J$ is Cohen-Macaulay iff $\normalfont{\textrm{in}}(J)$ is Cohen-Macaulay iff $\Delta_J$ is Cohen-Macaulay. 
\end{Th}

It turns out that $\Delta_I$ is Cohen-Macaulay if only if the ladders of zeroes, $L_1$ and $L_2$, are triangles, or $L_1$ and $L_2$ are small. We will show later that these configurations are sufficient for $\Delta_I$ to not only be Cohen-Macaulay, but also vertex decomposable. In light of Theorem \ref{Th-ICM-iff-InICM}, this characterizes the configurations for which $I$ is Cohen-Macaulay. \\

Before proceeding we recall Reisner's Criterion for when a complex is Cohen-Macaulay.

\begin{Th}\label{Th-Reisner-Criterion}$(${\cite[8.1.6]{Herzog2011}}$)$
Let $\Delta$ be a complex over a field $k$. Then $\Delta$ is Cohen-Macaulay if and only if for every face $F \in \Delta$ and for every $i < \dim \link_{F}(\Delta)$, we have $$\widetilde{H}_i(\link_{F}(\Delta); k) = 0.$$

\end{Th}

The following are some useful applications of Reisner's Criterion.

\begin{Cor}\label{Cor-Reisner-Cor}
Let $\Delta$ be a complex. Then

\begin{enumerate}
    \item If $\Delta$ is Cohen-Macaulay, then for every face $F \in \Delta$, $\link_F(\Delta)$ is Cohen-Macaulay.
    \item Let $F$ be a simplex whose vertex set is disjoint from that of $\Delta$, then $\Delta * F$ is Cohen-Macaulay if and only if $\Delta$ is Cohen-Macaulay.
\end{enumerate}
\end{Cor}
%Technical note: Miller-Sturmfels version of Reisner's criteron is not quite match what we used. However it agrees with what we used when the complex is pure. 
We now provide necessary conditions for $\Delta_I$ to be Cohen-Macaulay.
\begin{Prop}\label{Prop-Nec-CM}

Suppose $I \neq 0$ and $\Delta_I$ is Cohen-Macaulay, let $X'$ be the matrix obtained from zeroing $U_{r-1}$ and $D_{r-1}$, then 

\begin{enumerate}
    \item $X'$ is unpinched or $X'$ is a diagonal matrix. 
    \item $L_1$ is a triangle or contained in a triangle of size $r-1\times r-1$, and $L_2$ is a triangle or contained in a triangle of size $r-1\times r-1$.
\end{enumerate}

\end{Prop}

\begin{proof}

Let $\Delta'$ be the complex for $X'$. By Proposition \ref{Prop-Contain-Triangle}, we have that $\Delta_I = \Delta' * (U_{r-1}\cup D_{r-1})$, hence $\Delta'$ is Cohen-Macaulay by Corollary \ref{Cor-Reisner-Cor} (2). \\

$(1)$ Suppose $r=2$. Let $A_i$ be the diagonal blocks in the unpinched decomposition of $X'$. We have that $\Delta' = \bigsqcup \Delta_{I_2(A_i)}$. If $\dim \Delta' > 0$, then since $\Delta'$ is connected, there is only one $A_i$. If $\dim \Delta' = 0$, then $\dim \Delta_{I_2(A_i)} = 0$, hence $X'$ is a diagonal matrix. Now for $r > 2$, we let $E$ be the first row and first column and consider $\Delta'' = \link_{E}(\Delta')$. $\Delta''$ is the complex of the $r-1$ minors of a submatrix $Y$, obtained by deleting the first row and column of $X'$. By Corollary \ref{Cor-Reisner-Cor}, $\Delta''$ is Cohen-Macaulay, hence by induction, $Y$ is unpinched or a diagonal matrix. \\

Suppose $Y$ is a diagonal matrix. As $r > 2$ and $I \neq 0$, $Y$ is at least a $2 \times 2$ matrix, this implies that $X$ has a $1 \times 1$ block in its unpinched decomposition. Now by Proposition \ref{Prop-Purity}, $\Delta_{I_2(X')}$ is pure, hence if one block of $X$ is $1 \times 1$ then all blocks of $X$ are $1 \times 1$.\\

Now suppose $Y$ is unpinched. If $X$ is pinched, then $X$ has block form $$\begin{bmatrix} A & 0 \\ 0 & Y \end{bmatrix},$$ where $A$ is a $1 \times 1$ matrix. But by the same reasoning above, this would mean that all of the blocks of $X$ are $1 \times 1$ which implies that $Y$ is a diagonal matrix. But since $Y$ is at least $2 \times 2$, $Y$ would be pinched, a contradiction.\\

$(2)$ We may replace $X$ by $X'$ and $\Delta$ by $\Delta'$ to assume $L_1$ and $L_2$ contain triangles of size $r-1\times r-1$.  We will show that $L_1$ and $L_2$ are triangles. We are done if $X'$ is a diagonal matrix, hence by $(1)$, we may assumed $X'$ is unpinched.\\

We proceed by induction on $r$. The base case is $r=2$. For contradiction, assume one of $L_1$ or $L_2$ is not a triangle. We may transpose the matrix to assume $L_1$ is not a triangle. Using Theorem \ref{Th-Reisner-Criterion}, it is sufficient to find a face $F$ such that $\link_{F}(\Delta)$ is disconnected of dimension bigger than $0$. This is equivalent to finding a face $F$ which is properly contained in exactly 2 facets $F_1$ and $F_2$ such that $F_1 \cap F_2 = F$, and $\dim F_1 \setminus F > 0$ or $\dim F_2 \setminus F > 0 $.  By $(1)$, $X$ is unpinched, and as $L_1$ is not a triangle, the following submatrix must appear along the border of $L_1$, $$ \begin{bmatrix} * & * & * \\ 0 & 0 & * \\ 0 & 0 & * \end{bmatrix}.$$

We construct $F$ by taking any path from a corner of $L_2$ and ending at the top right entry of the submatrix. By the remark after Proposition \ref{Prop-Stairs}, $F$ satisfies condition $F_2$. $F$ and can be completed into a maximal stair by adding entries either to the left of, or below, the top right entry. Hence there is exactly $2$ ways to complete $F$ into a stair, since if we add the entry to the left, then we must add all entries to the left, similarly with the entry below. This yields two facets $F_1$ and $F_2$ containing $F$. It is clear that the dimension after deleting $F$ from both of them is at least $1$. \\

For the inductive step, let $E$ be the face consisting of the first row and column of $X$, and set $\Delta' = \link_{E}(\Delta_I)$. As before, $\Delta'$ is the complex associated to the $r-1$ minors of a submatrix $Y$, obtained from $X$ by deleting the first row and first column. By induction, both blocks of zeros of $Y$ are triangles. But now, we are done as we can repeat the argument with $E$ being the last row and last column.
\end{proof}

Before we proceed, we provide a definition.

\begin{Def}\label{Def-Vert-Decomp}  (\cite[16.41]{Miller2005})
A complex $\Delta$ is \textit{vertex decomposable} if $\Delta = \{ \emptyset \}$ or it is pure and has a vertex $v$ such that $\del_v(\Delta)$ and $\link_v(\Delta)$ are vertex decomposable.

\end{Def}

Next, we need a lemma about vertex decomposability and joins of complexes.

\begin{Lemma} \label{Lemma-Cone-Over-VD}$($\cite[2.4]{Provan1980}$)$
Let $\Delta$ and $\Sigma$ be complexes. Then $\Delta$ and $\Sigma$ are vertex decomposable  if and only if $\Sigma * \Delta$ is vertex decomposable.
\end{Lemma}

Note that the definition of vertex decomposability in \cite{Provan1980} is not the same as our definition. However, they have been shown to be equivalent, see \cite{Provan1977}. Observe that simplices are vertex decomposable. Hence we will apply, quite often, the above lemma with $\Sigma$ a simplex. \\

Proposition \ref{Prop-Contain-Triangle} tells us that $\Delta_I = (U_{r-1}\cup D_{r-1}) * \del_{U_{r-1}\cup D_{r-1}}(\Delta_I)$. If we zero the entries in $U_{r-1}$ and $D_{r-1}$, then we obtain a generic GD matrix $Y$ and we see that $\Delta_{I_r(Y)} = \del_{U_{r-1}\cup D_{r-1}}(\Delta_I)$. As we've seen from the Lemma \ref{Lemma-Cone-Over-VD}, forming a cone with a simplex preserves vertex decomposability. Hence in the context of vertex decomposability, we may employ the above lemma to assume that $L_1$ and $L_2$ contain triangles of size $r-1\times r-1$.\\

By \cite{Miller2005} for pure simplicial complexes,$$
\textrm{vertex decomposable} \Rightarrow \textrm{shellable} \Rightarrow \textrm{Cohen-Macaulay}.
$$

When $X$ is a generic matrix, $\Delta_I$ has already been shown, by numerous people, to be shellable \cite{Bjorner1980}, \cite{Conca2003}, \cite{Herzog1992}. As pure shellable complexes are Cohen-Macaulay \cite[13.45]{Miller2005}, the converse to Proposition \ref{Prop-Nec-CM} is already known in the setting where $L_1$ and $L_2$ are empty. When $L_1$ and $L_2$ are triangles, our complexes can be realized as a special case of higher order complexes, which were shown to be shellable in \cite[Section 7]{Bjorner1980}. \\

We will show that $\Delta_I$ is, in fact, vertex decomposable. We are unaware of any other proof that $\Delta_I$ is vertex decomposable when $L_1$ and $L_2$ are triangles, empty or otherwise.\\

The following theorem is a strong converse to Proposition \ref{Prop-Nec-CM}. 

\begin{Th}\label{Th-Triangle-VD}Suppose $L_1$ is a triangle or contained in a $r-1 \times r - 1$ triangle, and $L_2$ is also a triangle or contained in a $r-1 \times r - 1$ triangle. Then $\Delta_I$ is vertex decomposable.
\end{Th}

Before we proceed with the proof, we need to establish some notation and terminology. We will show that $\Delta_I$ is vertex decomposable by describing a recursive algorithm for choosing vertices that satisfy Definition \ref{Def-Vert-Decomp}. We will choose the vertices in such a way that we build a ladder in the top left corner. Let $L_1$ and $L_2$ be triangles of size $t_1$ and $t_2$, respectively. We say that $\mathcal{L}$ is a ladder of entries if for a non-increasing sequence of non-negative integers $a_1,...,a_{m - t_2}$, with $a_1 \leq n - t_1$, $\mathcal{L}$ consist of the first $a_k$ entries in the $k$-th column. We allow $\mathcal{L}$ to be empty. We will augment a ladder $\mathcal{L}$ with additional data: two subsets $\textbf{D}$ and $\textbf{L}$ that form a partition of $\mathcal{L}$. \\

Now given a non-empty set of non-zero entries $\mathbf{A} = \{x_1,...,x_k\}$ of $X$ and a subcomplex $\Delta$ of $\Delta_I$, we set $\link(\mathbf{A}, \Delta)$ to be the iterated link along $\mathbf{A}$, i.e., $$\link(\mathbf{A}, \Delta) = \link_{x_k}(\link_{x_{k-1}}( ... \link_{x_1}(\Delta))).$$
Notice that by \cite[1.1]{Provan1980}, $\link(\mathbf{A})$ does not depend on the ordering of the entries of $\mathbf{A}$ and hence is well defined. For convenience, we set $\link(\emptyset, \Delta) = \Delta$. Similarly, we define $\del(\mathbf{A})$, and set $\del(\emptyset, \Delta) = \Delta$. Now given $(\mathcal{L}, \mathbf{D}, \mathbf{L})$, we define $\Delta_I(\mathcal{L}) = \Delta_I(\mathcal{L}, \mathbf{D}, \mathbf{L}) = \link(\mathbf{L}, \del(\mathbf{D},\Delta_I))$. We note that by \cite[1.1]{Provan1980}, $\Delta(\mathcal{L})$ is the same as the iterated deletion and link along $\mathcal{L}$ in any order. If the underlying ideal of minors $I$ is understood, we will drop it as a subscript and write $\Delta(\mathcal{L})$. Before we prove the theorem, we need the following lemma.

\begin{Lemma} Suppose $L_1$ is a triangle or contained in a $r-1 \times r - 1$ triangle, and $L_2$ is also a triangle or contained in a $r-1 \times r - 1$ triangle. Given $(\mathcal{L}, \mathbf{D}, \mathbf{L})$ with $\mathbf{D} = \mathcal{L}$, $\Delta(\mathcal{L})$ is pure of the same dimension as $\Delta_I$.
\end{Lemma}

\begin{proof} By Lemma \ref{Lemma-Cone-Over-VD} and Proposition \ref{Prop-Contain-Triangle}, we may assume $L_1$ and $L_2$ are triangles of size at least $r-1\times r-1$. We induct on $k = |\mathcal{L}|$. If $k = 0$, then $\Delta(\mathcal{L}) = \Delta_I$ is pure by Theorem \ref{Th-Triangle-Dim}. \\

Now suppose $k > 0$. Let $v=x_{ab}$ be an entry of $\mathcal{L}$ such that $x_{a+1,b} \notin \mathcal{L}$ and $x_{a,b+1} \notin \mathcal{L}$. In other words, $v$ is an inner corner of $\mathcal{L}$. Consider $\mathcal{L'} = \mathcal{L} \setminus \{v\}$. By induction, $\Delta(\mathcal{L'})$ is pure, furthermore since $\dim \Delta(\mathcal{L'}) = \dim \Delta_I$, we have that the facets of $\Delta(\mathcal{L'})$ are facets of $\Delta$ that do not meet $\mathcal{L'}$. \\

It suffices to show that a facet of $\Delta(\mathcal{L})$ is a facet of $\Delta(\mathcal{L}')$. Let $F$ be a facet of $\Delta(\mathcal{L})$. Then $F$ is a facet of $\Delta(\mathcal{L}')$ or $F \cup \{v\}$ is a facet of $\Delta(\mathcal{L}')$. We may assume the later case, otherwise we are done. We will show that this case leads to a contradiction by constructing a face $G$ of $\Delta(\mathcal{L})$ so that $F \subsetneq G$. By choice of $v$, we have that $\{x_{a+1,b+1}, ..., x_{a+r-1, b+r-1} \}$ is completely outside of $\mathcal{L}$. As $F \cup \{v\}$ is a facet of $\Delta_I$, we may consider its tendrils $T_1,..,T_{r-1}$. Since there are $r-1$ many tendrils and they are disjoint there must be an element of $\{x_{ab}, ..., x_{a+r-1, b+r-1}\}$ outside of $F \cup \{v\}$. Let $x_{a+l,b+l}$ be such an element with the smallest row index, and for $j = 1,...,l$, we relabel the tendrils so that $T_j$ contains $x_{a+j-1,b+j-1}$. Now set $G = F \cup \{ x_{a+l,b+l}\}$. We will be done once we show that $G$ is a face of $\Delta_I$, and since $G$ does not meet $\mathcal{L}$, it must also be a face of $\Delta(\mathcal{L})$. \\

We will show that $G$ satisfies condition $F_r$. Suppose not, then $G$ contains a $r$-chain $x_1<...<x_r$. We may assume the chain contains $x_{a+l,b+l}$. Otherwise, the chain would be in $F$, but $F$, considered as a face of $\Delta_I$, satisfies condition $F_r$. Now removing $x_{a+l,b+l}$ from this chain yields a $r-1$ chain in $F \subset F \cup \{v\}$. Since the tendrils of $F \cup \{v\}$ are disjoint and satisfy condition $F_2$, this forces each tendril to contain a unique element of this $r-1$ chain. Hence $T_1,...,T_{l}$ contains a $l$-chain, which we relabel as $x_1,...,x_{l}$. \\

Notice that for $j=1,...,l$, either $x_j < x_{a+l,b+l}$ or $x_{a+l,b+l} < x_j$. The latter case leads to a contradiction as then $x_{a+j-1,b+j-1} < x_j$ is a $2$-chain in $T_j$. Hence $x_j < x_{a+l,b+l}$ for all $j$. Now $x_1,...,x_l,x_{a+l,b+l}$ forms an $l+1$ chain but the shape of $\mathcal{L}$ and the choice of $x_{ab}$ makes this impossible.
\end{proof}

Since the link of any pure complex is still pure, we immediately arrive at the following corollary.

\begin{Cor}\label{Cor-Link-Del-Pure}  Suppose $L_1$ is a triangle or contained in a $r-1 \times r - 1$ triangle, and $L_2$ is also a triangle or contained in a $r-1 \times r - 1$ triangle.  Given $(\mathcal{L}, \mathbf{D}, \mathbf{L})$ where $\mathbf{D}$ is a ladder shape, then $\Delta(\mathcal{L})$ is pure.

\end{Cor}

We establish a terminology before proceeding. We say that a non-zero entry $x_{ij}$ is a corner of a ladder $\mathcal{L}$ if $x_{ij} \notin \mathcal{L}$ but $x_{i-1,j} \in \mathcal{L}$ or $i = 1$, and $x_{i,j-1} \in \mathcal{L}$ or $j=1$.  Notice that each row can contain at most one corner of $\mathcal{L}$, hence we may order the corners of $\mathcal{L}$ in ascending order by their row index. Also notice that if $x$ is a corner of $\mathcal{L}$, then $\mathcal{L} \cup \{x\}$ is also a ladder.\\

Now since we plan to prove Theorem \ref{Th-Triangle-VD} inductively, we would like to strengthen the statement and prove a stronger statement below. Theorem \ref{Th-Triangle-VD} will follow.

\begin{Th} Let $X$ be a generic $n \times m$ GD matrix. Suppose $L_1$ is a triangle or contained in a $r-1 \times r - 1$ triangle, and $L_2$ is also a triangle or contained in a $r-1 \times r - 1$ triangle. Let $(\mathcal{L}, \mathbf{D}, \mathbf{L})$ be a ladder such that $\mathbf{D}$ is a ladder shape, and $\mathbf{L}$ consist of the first few consecutive corners of $\mathbf{D}$, then $\Delta(\mathcal{L})$ is vertex decomposable.

\end{Th}

\begin{proof}
We induct on $n+m$. The base case when $n + m =2$ is clear. Now suppose $n + m > 2$. Given a ladder $(\mathcal{L}, \mathbf{D}, \mathbf{L})$ in the form of the statement, we will describe a recursive algorithm for picking a vertex $v$ that satisfies Definition $\ref{Def-Vert-Decomp}$. \\

If $\mathbf{L}$ equals the set of all corners of $\mathbf{D}$, we stop. Otherwise, choose $v$ as the first corner of $\mathbf{D}$ not in $\mathbf{L}$. Now $(\mathcal{L} \cup \{v\}, \mathbf{D}\cup \{v\}, \mathbf{L})$ and $(\mathcal{L} \cup \{v\}, \mathbf{D}, \mathbf{L} \cup \{v \})$ are ladders satisfying the statement of the theorem, hence we may repeat the process with the two ladder shapes. \\

It is clear that the algorithm must terminate since there is only finitely many corner positions for any given ladder and the size of any ladder is bounded. Now in each step of the algorithm, $(\mathcal{L}, \mathbf{D}, \mathbf{L})$ satisfies the conditions of Corollary \ref{Cor-Link-Del-Pure}, hence $\Delta(\mathcal{L})$ is pure. It remains to show that the terminating ladder shapes are vertex decomposable. \\

Let $\mathcal{L}$ be a terminating ladder shape, i.e., all of the corners of $\mathbf{D}$ are in $\mathbf{L}$. If all of the non-zero entries in the first row are in $\mathbf{D}$, then $\Delta(\mathcal{L})$ is equal to $\Delta_{I'}({\mathcal{L}'})$ for some ladder $\mathcal{L}'$ in the matrix after removing the first row and $I'$ is the $r \times r$ minors of that matrix. Hence by induction, $\Delta(\mathcal{L})$ is vertex decomposable. A similar argument applies if all of the non-zero entries in the first column are in $\mathbf{D}$. \\

Now we may assume that not all of the non-zero entries in the first row and first column are in $\mathbf{D}$. Then $\mathbf{D}$ has a corner in the first row and a corner in the first column, which by assumption will both be in $\mathbf{L}$. Now every facet of $\Delta(\mathcal{L})$ comes from a facet of $\Delta_I$ that contains all of $\mathbf{L}$ and does not meet $\mathbf{D}$. Since $\mathbf{L}$ consist of the all of the corners of a ladder shape, for any facet of $\Delta(
\mathcal{L})$, there is a uniquely determined tendril, $T$, that goes through $\mathbf{L}$. Furthermore, $T$ is obtained by scraping. Hence after removing $T$, what remains can be considered as a facet of a submatrix $Y$, obtained by removing the first row. Hence $\Delta(\mathcal{L})$ is a cone over the complex of some ladder in $Y$. More precisely, consider $(\mathcal{L'} = (\mathcal{L} \cup T) \setminus \{\textrm{first row} \}, \mathbf{D}' = (\mathbf{D} \cup T) \setminus \{\textrm{first row} \}, \emptyset)$. Then $\Delta(\mathcal{L}) = (T\setminus\mathbf{L}) * \Delta_{I'}(\mathcal{L}')$, where ${\mathcal{L}}'$ is a ladder for the $r-1$ minors, $I'$, of $Y$. Hence $\Delta(\mathcal{L})$ is vertex decomposable by induction and Lemma \ref{Lemma-Cone-Over-VD}.
\end{proof}

Theorem \ref{Th-Triangle-VD} follows since by Proposition \ref{Prop-Contain-Triangle} and Lemma \ref{Lemma-Cone-Over-VD} we may assume $L_1$ and $L_2$ are triangles of size at least $r-1\times r-1$. 
Now, combining Proposition \ref{Prop-Nec-CM} and Theorem \ref{Th-Triangle-VD}, we achieve the following.

\begin{Cor}
$\Delta_I$ is Cohen-Macaulay if and only if $L_1$ is a triangle or contained in a $r-1 \times r - 1$ triangle, and $L_2$ is a triangle or contained in a $r-1 \times r - 1$ triangle.
\end{Cor}

\section{Multiplicity of $I_r(X)$}\label{Sec-Multiplicity}

The $f$-vector of a simplicial complex is the vector whose $j$th component is the number of faces of that complex of dimension $j$.   
With the above description of $\Delta_I$ by \cite[6.2.1]{Herzog2011}, one can compute the multiplicity of $I$ by counting the number of maximal $(r-1)$-stairs. We demonstrate the computation of $f_d$, where $d = \dim \Delta_I$, in the case that the ladders of zeroes, $L_1$ and $L_2$, are triangles of size $t_1$ and $t_2$, respectively, cf. \cite[3.5]{Herzog1992}. If $L_1$ and $L_2$ are not triangles, the computation becomes more complicated.

\begin{Def} (\cite[Section 2.7]{Stanley2011})
A \textit{lattice path} $P=(v_0,...,v_k)$ in $\mathbb{N}^2$ with steps $(-1,0),(0,1)$ is a collection of points such that $v_i-v_{i-1}= (-1,0) \text{ or } (0,1)$.\\

An \textit{$n$-path} $\textbf{P}$ of type $(\alpha,\beta,\gamma, \delta)\in (\mathbb{N}^n)^4$ is a collection of paths $\textbf{P}=(P_1,...,P_n)$ such that $P_i$ is a path from $(\beta_i,\gamma_i)$ to $(\alpha_i,\delta_i)$. $\textbf{P}$ is said to be non-intersecting if $P_i\cap P_j=\emptyset$ for $i\neq j$.
\end{Def}

\begin{Th}\label{Th-pathcount} $($\cite[2.7.1]{Stanley2011}$)$
Let $(\alpha,\beta,\gamma, \delta)\in (\mathbb{N}^n)^4$ such that for all $w\in S_n$, the set of non-intersecting $n$-paths of type $(w(\alpha),\beta,\gamma, w(\delta))$ is empty unless $w$ is the identity. Then the number of non-intersecting $n$-paths of type $(\alpha,\beta,\gamma, \delta)$ is
\[\det [h_{i,j}],\]
where $h_{i,j}$ counts the number of paths from $ (\beta_i,\gamma_i)$ to $ (\alpha_j,\delta_j)$.
\end{Th}

Write $(\alpha,\beta,\gamma, \delta)=(A,B)\in (\mathbb{N}^2)^2$ where $A=(\alpha,\delta)$ and $B=(\beta,\gamma)$. Then it becomes clear that the condition on permutations informally translates to: Any non-trivial permutation of $A$ would force any $n$-path of type $(A,B)$ to be intersecting.

\begin{Not}
Let $\Gamma(A,B)$ denote the number of non-intersecting $n$-paths of type $(A,B)=(\alpha,\beta,\gamma, \delta)$, and let $h(A_j,B_i)$ denote the number of paths from $B_i$ to $A_j$.
\end{Not}

Notice that \[h(A_j,B_i)=
\begin{cases} 
      {{ \alpha_j-\beta i+\gamma_i-\delta_j }\choose {\alpha_j-\beta_i}} ={{ \alpha_j-\beta i+\gamma_i-\delta_j }\choose {\gamma_i-\delta_j}} & \alpha_j-\beta_i\geq 0 \text{ and }\gamma_i-\delta_j \geq 0\\
      0 &  \textrm{otherwise} \\
       
   \end{cases}.
\]

\begin{Not}
Given an ordered collection $V=(v_1,...,v_n) \subset \mathbb{N}^2$ and a set $N=\{k_1,...,k_m\}$ with $1\leq k_1 < k_2<...<k_m\leq n$, write $V_N:=(v_{k_1},...,v_{k_m})$.
\end{Not}

\begin{Lemma}
Let $L_1$ and $L_2$ be triangles of size $t_1$ and $t_2$ respectively, and $d=\dim \Delta_I$. Set $T_1=\max\{r-2,t_1\}$ and $T_2=\max\{r-2, t_2\}$ then \[ f_d(\Delta_I)=\sum_{\substack{{1\leq a_1<...<a_{r-1}\leq T_1+1}\\{1\leq b_1<...<b_{r-1}\leq T_2+1} }}
   \det\left[\begin{cases}
 {{n+m-T_1-T_2-2} \choose { m-1-T_2+b_j-a_i } }    &  T_2-m+1\leq b_j-a_i \leq n-1-T_1 \\
   0    & \normalfont{\textrm{otherwise } }
 \end{cases} \right]
.
\]
\end{Lemma}

 \begin{proof}
 Set $\Delta=\Delta_I$. Joining a complex with a simplex does not change the top value of the $f$-vector, only its index. Hence by Proposition \ref{Prop-Contain-Triangle}, we may replace $\Delta$ by a complex $\Delta'$, where $\Delta '$ is the complex for a generic $n\times m$ GD matrix with $L_1$, a size $T_1$ triangle, and $L_2$, a size $T_2$ triangle. In this setting, it follows from Theorem \ref{Th-Triangle-Dim} that a maximal $(r-1)$-stair is a disjoint union of $r-1$ disjoint maximal stairs, hence a $(r-1)$-stair is a $(r-1)$-path.\\
 
 Set $A_a=(n-T_1-1+a,a)$ for $1\leq a \leq T_1+1$, and set $B_b=(b,m-T_2-1+b)$ for $1\leq b \leq T_2 +1$. These are the corners of $L_1$ and $L_2$, respectively.\\
 
 Then each facet of $\Delta'$ correspond to a $(r-1)$-path of type $(A_J,B_K)$, where $J={a_1,...,a_{r-1}}$ with $1\leq a_1 <a_2<...<a_{r-1}\leq T_1 +1$, and $K={b_1,...,b_{r-1}}$ with $1\leq b_1 <b_2<...<b_{r-1}\leq T_2 +1$, hence \[f_d=\sum_{\substack{J={1\leq a_1<...<a_{r-1}\leq T_1+1}\\K={1\leq b_1<...<b_{r-1}\leq T_2+1} }} \Gamma(A_J,B_K) .\] 

 Now it is clear that $(A_J,B_K)$ satisfies the assumption of Theorem \ref{Th-pathcount} hence  $\Gamma(A_J,B_K)=\det[h((A_J)_i,(B_K)_j)]=\det[h( (n-T_1-1+a_i,a_i),(b_j,m-T_2-1+b_j))] $. By the above formula, we have that \[
 h( (n-T_1-1+a_i,a_i),(b_j,m-T_2-1+b_j))\]
 \[
 =\begin{cases}
 {{(n-T_1-1+a_i-b_j)+(m-T_2-1+b_j-a_i)} \choose { m-T_2-1+b_j-a_i} }  &  T_2-m+1\leq b_j-a_i \leq n-1-T_1\\
    0   & \textrm{otherwise } 
     \end{cases} \]
 
 \[
 =\begin{cases}
 {{n+m-T_1-T_2-2} \choose { m-1-T_2+b_j-a_i } }     &  T_2-m+1\leq b_j-a_i \leq n-1-T_1 \\
   0    & \text{otherwise } 
 \end{cases} .\]
 \end{proof}

\begin{Rmk}
In \cite{Herzog1992}, since the authors are only concerned with generic matrices and do not need to avoid zeros, they choose their starting points and ending points to be the last $r-1$ entries on the first row and column, respectively. Specializing our proof to the generic case will yield a different matrix with the same determinant. This is an artifact of the different counting methods as can be seen in the example below. 
\\

Let $m=n=4$, $t_1=t_2=0$, and $r=3$. That is to say, we are considering the $3 \times 3$ minors of a generic $4\times 4$ matrix. Then we have that \[
\renewcommand\arraystretch{1.2}
f_d=\det  \begin{bmatrix} \binom{4}{2} & {4\choose 3} \\ {4\choose 3} & {4\choose 2} \end{bmatrix}= \det \begin{bmatrix} 6 & 4\\ 4 & 6 \end{bmatrix} =20,\]
compared to the the formula in \cite[3.5]{Herzog1992}:
\[
\renewcommand\arraystretch{1.2}
f_d=\det \begin{bmatrix} {6\choose 3} & {5\choose 3} \\ {5\choose 2} & {4\choose 2} \end{bmatrix}= \det \begin{bmatrix} 20 & 10\\ 10 & 6 \end{bmatrix} =20.\]
\end{Rmk}

  \bibliographystyle{amsalpha}
  \bibliography{GD}

\providecommand{\bysame}{\leavevmode\hbox to3em{\hrulefill}\thinspace}
\providecommand{\MR}{\relax\ifhmode\unskip\space\fi MR }
\providecommand{\MRhref}[2]{%
  \href{http://www.ams.org/mathscinet-getitem?mr=#1}{#2}
}
\providecommand{\href}[2]{#2}
\begin{thebibliography}{GM82}

\bibitem[BC03]{Conca2003}
W.~Bruns and A.~Conca, \emph{Gröbner bases and determinantal ideals},
  Commutative Algebra, Singularities and Computer Algebra (Dordrecht) (Jürgen
  Herzog and Victor Vuletescu, eds.), Springer Netherlands, 2003, p.~9–66.

\bibitem[Bj{\"o}80]{Bjorner1980}
A.~Bj{\"o}rner, \emph{Shellable and {C}ohen-{M}acaulay partially ordered sets},
  Trans. Amer. Math. Soc. \textbf{260} (1980).

\bibitem[CV20]{Conca2020}
A.~Conca and M.~Varbaro, \emph{Square-free {G}röbner degenerations},
  Inventiones Mathematicae \textbf{221} (2020), 713--730.

\bibitem[EN62]{EagonNorthcott}
J.~A. Eagon and D.~G. Northcott, \emph{Ideals defined by matrices and a certain
  complex associated with them}, Proceedings of the Royal Society of London.
  Series A, Mathematical and physical sciences \textbf{269} (1962), no.~1337,
  188--204.

\bibitem[GM82]{Giusti1982}
M.~Giusti and M.~Merle, \emph{Singularites isolees et sections planes de
  varietes determinantielles}, Algebraic geometry (La R{\'a}bida, 1981) (1982),
  103--118.

\bibitem[Gor07]{Gorla2007}
E.~Gorla, \emph{Mixed ladder determinantal varieties from two-sided ladders},
  Journal of Pure and Applied Algebra \textbf{211} (2007), 433--444.

\bibitem[HH11]{Herzog2011}
J.~Herzog and T.~Hibi, \emph{Monomial ideals}, 1st ed. 2011. ed., Graduate
  Texts in Mathematics, 260, Springer London, London, 2011.

\bibitem[HT92]{Herzog1992}
J.~Herzog and N.~Trung, \emph{Gröbner bases and multiplicity of determinantal
  and pfaffian ideals}, Advances in Mathematics \textbf{96} (1992), 1--37.

\bibitem[J{\'o}z78]{Jozefiak}
T.~J{\'o}zefiak, \emph{Ideals generated by minors of a symmetric matrix},
  Commentarii Mathematici Helvetici \textbf{53} (1978), no.~1, 595--607.

\bibitem[MS05]{Miller2005}
E.~Miller and B.~Sturmfels, \emph{Combinatorial commutative algebra}, 1st ed.
  2005. ed., Springer New York, 2005, Includes bibliographical references (p.
  379-395) and index.

\bibitem[PB80]{Provan1980}
J.~S. Provan and L.~J. Billera, \emph{Decompositions of simplicial complexes
  related to diameters of convex polyhedra}, Mathematics of Operations Research
  \textbf{5} (1980), 576--594.

\bibitem[Pro77]{Provan1977}
J.~S. Provan, \emph{Decompositions, shellings, and diameters of simplical
  complexes and convex polyhedra.}, 1977.

\bibitem[Sta11]{Stanley2011}
R.~P. Stanley, \emph{Enumerative combinatorics}, 2 ed., Cambridge Studies in
  Advanced Mathematics, vol.~1, Cambridge University Press, 2011.

\end{thebibliography}

\end{document}